\documentclass{amsart}

\usepackage[utf8]{inputenc}
\usepackage{amssymb,amstext}
\usepackage{amsmath}
\usepackage{graphicx,color}

\newtheorem{theorem}{Theorem}

\newtheorem{proposition}[theorem]{Proposition}

\newtheorem{remark}[theorem]{Remark}
\newtheorem{definition}[theorem]{Definition}

\newtheorem{example}[theorem]{Example}

\numberwithin{equation}{section}
\numberwithin{theorem}{section}

\newcommand{\R}{\mathbb{R}}
\newcommand{\N}{\mathbb{N}}
\newcommand{\1}{\ell^1}
\newcommand{\2}{\ell^2}
\newcommand{\3}{\ell^\infty}

\newcommand{\xdagtil}{\widetilde{x}^\dagger}

\def\xdag{x^\dagger}
\def\xad{{x_\alpha^\delta}}
\def\yd{y^\delta}
\def\sgn{{\rm sgn}}
\begin{document}


  \author{Jens Flemming}
  \author{Bernd Hofmann}
  \author{Ivan Veseli\'c}
  \address{Technische Universit\"at Chemnitz, Fakult\"at f\"ur Mathematik, 09107 Chemnitz, Germany}
  \title{On $\ell^1$-regularization in light of Nashed's ill-posedness concept}
\begin{abstract}
Based on the powerful tool of variational inequalities, in recent papers convergence rates results on $\mathbf{\ell^1}$-regularization for ill-posed inverse problems
have been formulated in infinite dimensional spaces under the condition that the sparsity assumption slightly fails,
but the solution is still in $\ell^1$.
In the present paper we improve those convergence rates results and apply them to the Ces\'aro operator equation in $\2$ and to specific denoising problems.
Moreover, we formulate in this context relationships between Nashed's types of ill-posedness and mapping properties like
compactness and strict singularity.
\end{abstract}


\maketitle

\section{Introduction} \label{s1}

Variational sparsity regularization based on $\1$-norms became of significant interest in the past ten years with respect to inverse problems applications, e.g.~in imaging (cf., e.g.,~\cite{Scherzetal09}), but also with respect to the progress in
regularization theory for the treatment of ill-posed operator equations in infinite dimensional Hilbert and Banach spaces (cf., e.g.,~\cite{BeBu11,BreLor09,Grasm10,Grasmei11,Lorenz08,RamRes10}). Moreover, with focus on sparsity, the use of $\ell^1$-regularization can be motivated for specific classes of well-posed problems, too (cf., e.g.,~\cite{CanRomTao06}). Based on the powerful tool of variational inequalities (also called variational source conditions), in \cite{BurFleHof13} convergence rates results on $\mathbf{\ell^1}$-regularization for linear ill-posed operator equations have been formulated in infinite dimensional spaces under the condition that the sparsity assumption slightly fails, but the solution is still in $\ell^1$. In the present paper, we improve those results and illustrate the improvement level with respect to the associated convergence rates for the Ces\'aro operator equation in $\2$ and for specific denoising problems. Since the variational inequality approach requires injectivity of 
the forward operator (cf.~\cite[Proposition~5.6]{BurFleHof13}), we restrict all considerations in this paper to {\sl injective} linear forward operators which also ensure uniquely determined solutions for the corresponding linear operator equations. The focus on injectivity is also motivated
by the fact that the ill-posedness concept of Section~\ref{s4} suggested by M.~Z.~Nashed (cf.~\cite{Nashed86}) would require substantial technical refinements in a general Banach space setting if non-injective operators were included.

Let $\widetilde A: X \to Y$  be an {\sl injective} and {\sl bounded linear operator} mapping between an {\sl infinite dimensional separable Hilbert space} $X$ and an infinite dimensional {\sl Banach space} $Y$ with norms $\|\cdot\|_X$ and $\|\cdot\|_Y$, respectively. We are searching for the uniquely determined solution $\xdagtil \in X$ of the linear operator equation
\begin{equation}
  \label{eq:tildeopeq}
  \widetilde A\, \widetilde x\,=\,y, \qquad \widetilde x \in X,\quad y \in \mathcal{R}(\widetilde A),
\end{equation}
where $\mathcal{R}(\widetilde A)$ denotes the range of $\widetilde A$. Typically, instead of $y$ only noisy data $\yd \in Y$ are available. In this context, we consider the deterministic noise model
\begin{equation}  \label{eq:noise}
   \|\yd-y\|_Y \le \delta
\end{equation}
with given noise level $\delta>0$.

If the operator $\widetilde A$ is {\sl normally solvable}, i.e.\ its range is a closed subset in $Y$, then solving
the equation (\ref{eq:tildeopeq}) is a {\sl well-posed problem}. Consequently, for injective $\widetilde A$ the inverse ${\widetilde A}^{-1}: \mathcal{R}(\widetilde A) \subset Y \to X$ exists and is also a bounded linear operator.
If, on the other hand, the range of $\widetilde A$ is not closed, the inverse ${\widetilde A}^{-1}$
is an unbounded linear operator and solving the equation (\ref{eq:tildeopeq}) is an {\sl ill-posed problem}. This means that small perturbations in the right-hand side may lead to arbitrarily large error in the solution. Then regularization methods are required for obtaining stable approximate solutions to equation (\ref{eq:tildeopeq}).

As usual we consider in the sequel the Banach spaces $\ell^q$, $1 \le q < \infty$, and $\3$ of infinite sequences of real numbers with finite norms
$$\|x\|_{\ell^q}:=\left( \sum \limits_{k=1}^\infty |x_k|^q\right)^{1/q}\quad\text{and}\quad \|x\|_{\3}:=\sup_{k \in \N} |x_k|.$$
The Banach space $c_{\scriptscriptstyle 0}$ consists of
the real sequences $(x_k)_{k \in \N}$ with $\lim \limits_{k \to \infty} |x_k|=0$ and is also equipped with the norm $\|x\|_{c_0}:=\sup_{k \in \N} |x_k|$.  Moreover, by $\ell^0$ we denote  the
set of sparse sequences, where a sequence is \emph{sparse} if only a finite number of components is not zero.
For such sequences the number of nonzero components is given by
$$\|x\|_{\ell^0}:=\sum_{k=1}^\infty|\sgn(x_k)|\quad\text{with}\quad
\sgn(t):=\begin{cases}1,&t>0,\\0,&t=0,\\-1,&t<0.\end{cases}$$

Throughout this paper we fix an orthonormal basis $\{u^{(k)}\}_{k \in \N}$ in the Hilbert space $X$.
By $x=(x_k)_{k\in\N}\in\2$ we denote the infinite sequence of corresponding Fourier coefficients of $\widetilde{x}$, i.e.\
$$\widetilde x=\sum \limits_{k=1}^\infty x_k u^{(k)}.$$
The synthesis operator  $L: \1 \to X$ defined as $Lx:=\widetilde x$ is an injective bounded linear operator.
Note that this operator can be extended to $\2$, but in our setting we define it only on $\1$.

The focus of our studies is on \emph{almost sparse} solutions $\xdagtil$ to equation (\ref{eq:tildeopeq}). This means that only a finite number of coefficients $x_k^\dagger$ from the infinite sequence $\xdag=(\xdag_k)_{k \in \N}$ is relevant.
Here we do not require strict sparsity, $\xdag \in \ell^0$, but we allow an infinite number of nonzero coefficients if they decay fast enough.
Precisely, we assume $\xdag \in \1$ throughout this paper.

Introducing the operator $A:=\widetilde A \circ L: \1 \to Y$ our goal is to recover the solution $\xdag\in\1$ of
\begin{equation}
  \label{eq:opeq}
  Ax\,=\,y, \qquad x \in \1,\quad y \in \mathcal{R}(A),
\end{equation}
from noisy data $\yd \in Y$ satisfying (\ref{eq:noise}).
The following proposition shows that solving this equation is always an ill-posed problem, even if
the original equation (\ref{eq:tildeopeq}) is well-posed.

For the proof and for further reference we note that the synthesis operator $L$ is a composition $L=U \circ \mathcal{E}_2$ of the embedding operator
$\mathcal{E}_2:\1\to\2$ and the Riesz isomorphism $U:\2\to X$. Thus, the operator $A$ can be written as a composition
\begin{equation} \label{eq:three}
A=\widetilde A \circ U \circ \mathcal{E}_2
\end{equation}
of three injective bounded linear operators.

\begin{proposition} \label{pro:alwaysill}
The range of $A$ is not closed.
\end{proposition}
\begin{proof}
Assume that $\mathcal{R}(A)$ is closed. The full preimage of $\mathcal{R}(A)$ with respect to $\widetilde A$ is $\mathcal{R}(L)$.
Thus, $L$ has closed range, too. Looking at the composition (\ref{eq:three}), the full preimage of $\mathcal{R}(L)$ with
respect to $U$ is $\mathcal{R}(\mathcal{E}_2)=\1$. Consequently, $\1$ would be a closed subspace of $\2$. Since $\1$ is
dense in $\2$ this yields the contradiction $\1=\2$.
\end{proof}

The proposition shows that equation (\ref{eq:opeq}) requires regularization in order to obtain stable approximate solutions.
For this purpose we use a variant of variational regularization, called $\1$-regularization, where regularized solutions, denoted by $\xad$, are minimizers of the extremal problem
\begin{equation}
  \label{eq:TikBanach}
 \frac{1}{p}\|Ax-y^\delta\|_Y^p + \alpha \,\|x\|_{\1}
 \to \min, \quad \mbox{subject to} \quad x  \in \1.
\end{equation}
Here, $1 < p < \infty$ is some exponent and $\alpha>0$ is a regularization parameter. This regularization parameter is chosen in an appropriate manner, a priori as $\alpha=\alpha(\delta)$ depending on the noise level $\delta$, or  a posteriori as $\alpha=\alpha(\delta,\yd)$  depending also on the present regularized solution $\yd$ (for details see Sections~\ref{s2} and \ref{s3} below).

We are interested in error estimates
\begin{equation} \label{eq:genrates-1}
\|\xad-\xdag\|_{\1} \le C_{\xdag}\,\varphi(\delta)\qquad \mbox{for all} \qquad 0 < \delta \le \overline \delta,
\end{equation}
where the positive constant $C_{\xdag}$ may depend on the solution $\xdag$ but not on the noise level $\delta>0$. The estimates (\ref{eq:genrates-1})
can be interpreted as convergence rates
\begin{equation} \label{eq:genrates1}
\|\xad-\xdag\|_{\1}= O(\varphi(\delta)) \qquad \mbox{as} \qquad  \delta \to 0
\end{equation}
with rate functions $\varphi$ which are concave index functions.
Following \cite{MatPer03} and \cite{HofMat07} we call $\varphi: (0,\infty) \to (0,\infty)$ an index function if it is a continuous and strictly
increasing function with $\lim \limits_{t \to +0} \varphi(t)=0$.

The article is organized as follows: in the next section we briefly summarize results on existence, stability and convergence
of $\ell^1$-regularized solutions. Section~\ref{s3} contains the main theorem of this paper which improves the result of Theorem~5.2 from \cite{BurFleHof13} and can lead to better convergence rates.
Moreover, Section~\ref{s4} provides some insight into the interplay between Nashed's ill-posedness concept and mapping properties of the forward operator like compactness and strict singularity.
In the final Section~\ref{sec:denoise} we apply our findings to a problem of denoising type.

\section{Existence, stability and convergence of $\ell^1$-regularized solutions}    \label{s2}

From the general theory of Tikhonov regularization (cf.,~e.g.,~\cite[Section~3]{HKPS07}, \cite[Sections~4.1.1 and 4.1.2]{Schusterbuch12} and \cite[Section~3.1]{ItoJin14}) one can
infer the existence and stability of
$\ell^1$-regularized solutions $\xad$ as well as its convergence for $\delta \to 0$ to the uniquely determined solution $\xdag$ of equation (\ref{eq:opeq}) for appropriate choices of the regularization parameter $\alpha>0$.
For this purpose we summarize the results of Proposition~2.8 and Remark~2.9 from \cite{BurFleHof13} in the following proposition taking into account that the separable Banach space $c_0$ is a predual space of $\ell^1$,
that the operator $A: \1 \to Y$ is sequentially weak*-to-weak continuous, and that $\1$ satisfies the weak* Kadec-Klee property (for a proof see,~e.g.,~\cite[Lemma~2.2]{BoHo13}).
The weak*-to-weak continuity of $A$ can be shown along the lines of the proof of Lemma~2.7 in \cite{BurFleHof13} when $\mathcal{R}(A^*) \subset c_0$ is valid for the adjoint operator $A^*: Y^* \to \3$ to $A$.
In our setting,  $\mathcal{R}(A^*) \subset c_0$ is a consequence of the weak convergence $\widetilde A \,u^{(k)} \rightharpoonup 0$ in $Y$ for the prescribed orthonormal system $\{u^{(k)}\}_{k \in \N}$ in the Hilbert space $X$, which implies the weak convergence of formula
(\ref{eq:weakA}) in Remark~\ref{rem:distinguish} below such that the corresponding part of the proof of \cite[Proposition~2.4]{BurFleHof13} applies even if (\ref{eq:strongA}) is violated.   

\begin{proposition} \label{pro:exist}
For all $1<p<\infty$, $\alpha>0$ and $y^\delta \in Y$ there exist uniquely determined minimizers $\xad \in \1$ of the extremal problem (\ref{eq:TikBanach}). These $\1$-regularized solutions are always sparse, i.e.~they satisfy
$$\xad \in  \ell^0.$$
Furthermore, they are always stable with respect to the data, i.e., small perturbations in $y^\delta$ in the norm topology of $Y$ lead only to small changes in $\xad$ with respect to the $\1$-norm.
\par
If $\delta_n \to 0$ and if the regularization parameters $\alpha_n=\alpha(\delta_n,y^{\delta_n})$ are chosen such that $$\alpha_n \to 0 \qquad \mbox{and} \qquad \frac{\delta_n^p}{\alpha_n}\to 0 \qquad \mbox{as} \qquad n \to \infty,$$
then
\begin{equation} \label{eq:normconvergence}
\lim \limits_{n \to \infty} \|x_{\alpha_n}^{\delta_n}-\xdag\|_{\1}\,=\,0\,.
\end{equation}
\end{proposition}

\section{Improved convergence rates}    \label{s3}

The norm convergence (\ref{eq:normconvergence}) can be arbitrarily slow.
In order to obtain convergence rates, also for the $\ell^1$-regularization with regularized solutions $\xad$ defined as minimizers to problem (\ref{eq:TikBanach}),
a link condition between the smoothness of the solution $\xdag$ to (\ref{eq:opeq}) and the forward operator $A$ is required. From the studies and results of the recent paper \cite{BurFleHof13} we immediately derive the following theorem, where this link condition is a range condition imposed on all unit sequences $e^{(k)}=(0,\ldots,0,1,0,\ldots)$, $k \in \N$, with respect to the adjoint operator $A^*$.
We mention here that $\{e^{(k)}\}_{k \in \N}$ represents a Schauder basis in the Banach spaces $\ell^q$ for all $1 \le q <\infty$.

\begin{theorem} \label{thm:BFH}
Let the operator $A\colon \ell^1 \to Y$ from equation (\ref{eq:opeq})
be such that there exist elements $f^{(k)} \in Y^*$, $k \in \N$, satisfying the range conditions
\begin{equation} \label{eq:allrange}
e^{(k)}\,=A^*f^{(k)}.
\end{equation}
Then a variational inequality
\begin{equation} \label{eq:vi}
\|x-\xdag\|_{\1}\le \|x\|_{\1}-\|\xdag\|_{\1}+ \varphi_1(\|Ax-A\xdag\|_Y) \qquad \mbox{for all} \quad x \in \1
\end{equation}
is valid for the concave index function
\begin{equation} \label{eq:phiBFH}
\varphi_1(t)=2\inf_{n\in\mathbb{N}}\left(\sum_{k=n+1}^\infty\vert x^\dagger_k\vert
+t\sum_{k=1}^n\|f^{(k)}\|_{Y^*}\right).
\end{equation}
This yields the convergence rate
\begin{equation} \label{eq:genrates}
\|\xad-\xdag\|_{\1}= O(\varphi_1(\delta)) \qquad \mbox{as} \qquad \delta \to 0
\end{equation}
for $\1$-regularized solutions $\xad$ and for the uniquely determined solution $\xdag \in \1$ of equation (\ref{eq:opeq}) provided that the regularization parameter
$\alpha=\alpha(\delta,\yd)$ is chosen
appropriately, e.g.\ according to the discrepancy principle
\begin{equation} \label{eq:discr}
\tau_1 \delta \le
 \|Ax^\delta_{\alpha(\delta,\yd)}-\yd\|_Y \le \tau_2 \delta
\end{equation}
for prescribed values $1 < \tau_1 \le \tau_2<\infty$.
\end{theorem}

\begin{remark} \label{rem:curious} {\rm
For more details concerning the consequences of variational inequalities and the role of the choice of the regularization parameter for obtaining convergence rates in regularization  we refer, for example, to
\cite{HKPS07} and \cite{AnzHofMat14,BoHo10,BoHo13,Flemmingbuch12,Grasm10,HofMat12}.
Making use of Gelfand triples it was shown in \cite{AnzHofRam13} that, for a wide range of applied inverse problems, the forward operators $A$ are such that link conditions of the form (\ref{eq:allrange}) apply for all $e^{(k)},\;k \in \N$. On the other hand, the paper \cite{FleHeg14} gives counterexamples where (\ref{eq:allrange}) fails for specific operators $A$, but alternative link conditions presented there can compensate this deficit.
}\end{remark}

We improve the convergence rate obtained in the theorem above as follows:

\begin{theorem}\label{thm:improved}
Theorem \ref{thm:BFH} remains true if $\varphi_1$ is replaced by $\varphi_2$ with
\begin{equation} \label{eq:phiimproved}
\varphi_2(t)=2\inf_{n\in\mathbb{N}}\left(\sum_{k=n+1}^\infty\vert x^\dagger_k\vert
+t\sup_{\substack{a_k\in\{-1,0,1\}\\k=1,\ldots,n}}\left\Vert\sum_{k=1}^n a_k f^{(k)}\right\Vert_{Y^*}\right).
\end{equation}
\end{theorem}

\begin{proof}
From \cite[Lemma~5.1]{BurFleHof13} we know that
\begin{equation}
\label{eq:Pythagoras}
\|x-\xdag\|_{\1}-\|x\|_{\1}+\|\xdag\|_{\1} \le 2\left(\sum_{k=n+1}^\infty\vert x^\dagger_k\vert+\sum \limits_{k=1}^n |x_k-\xdag_k| \right).
\end{equation}
Observing
\begin{align*}
|x_k-\xdag_k|
&=\bigl(\sgn(x_k-\xdag_k)\bigr)(x_k-\xdag_k)\\
&=\bigl(\sgn(x_k-\xdag_k)\bigr)\langle e^{(k)},x-\xdag\rangle_{\3 \times \1}\\
&=\bigl(\sgn(x_k-\xdag_k)\bigr)\langle f^{(k)},Ax-A\xdag\rangle_{Y^\ast \times Y}
\end{align*}
the second sum on the right-hand side can be estimated above by
\begin{equation}
\label{eq:restricted-injectivity}
\begin{aligned}
\sum_{k=1}^n |x_k-\xdag_k|
&=\sum_{k=1}^n \bigl(\sgn(x_k-\xdag_k)\bigr)\langle f^{(k)},Ax-A\xdag\rangle_{Y^\ast \times Y}\\
&=\left\langle\sum_{k=1}^n \bigl(\sgn(x_k-\xdag_k)\bigr) f^{(k)},Ax-A\xdag\right\rangle_{Y^\ast \times Y}\\
&\leq\left\Vert\sum_{k=1}^n \bigl(\sgn(x_k-\xdag_k)\bigr) f^{(k)}\right\Vert_{Y^\ast}\|Ax-Ax^\dagger\|_Y\\
&\leq\|Ax-Ax^\dagger\|_Y\sup_{\substack{a_k\in\{-1,0,1\}\\k=1,\ldots,n}}\left\Vert\sum_{k=1}^n a_k f^{(k)}\right\Vert_{Y^*}
\end{aligned}
\end{equation}
and taking the infimum over all $n\in\N$ in the resulting inequality yields the variational inequality (\ref{eq:vi}).
\end{proof}

To understand the difference between $\varphi_1$ and $\varphi_2$ it may be helpful to note that $\varphi_1=\varphi_2$  if all $f^{(k)}$ are pairwise collinear. On the other hand, in the particular case that $Y^*=Y=\ell^2$
and $f^{(k)} =e^{(k)}$, we have
$\sum_{k=1}^n\|f^{(k)}\|_{Y^*}=n$ while
$\sup_{\substack{a_k\in\{-1,0,1\}\\k=1,\ldots,n}}\left\Vert\sum_{k=1}^n a_k f^{(k)}\right\Vert_{Y^*}=\sqrt n$.

In Example~\ref{ex:cesaro} and in Section~\ref{sec:denoise} we will show that the improved index function $\varphi_2$ yields better convergence rates for
some equations than the original function $\varphi_1$. The convergence rate result obtained in
\cite{FleHeg14} can be improved in a similar way.

\begin{example}[H\"older rates]\label{ex:hoelderrates} {\rm
If
\begin{equation}
\label{eq:Holeder-rates}
\sum_{k=n+1}^\infty\vert x^\dagger_k\vert\leq K_1\,n^{-\mu}\quad\text{and}\quad
\sup_{\substack{a_k\in\{-1,0,1\}\\k=1,\ldots,n}}\left\Vert\sum_{k=1}^n a_k f^{(k)}\right\Vert_{Y^*}\leq K_2\,n^\nu
\end{equation}
for $K_1,K_2\geq 0$ and $\mu,\nu>0$, then the convergence rate obtained in Theorem~\ref{thm:improved} is
\begin{equation} \label{firstrate}
\|\xad-\xdag\|_{\1}= O\bigl(\delta^{\frac{\mu}{\mu+\nu}}\bigr) \qquad \mbox{as} \qquad \delta \to 0
\end{equation}
(cf.\ \cite[Example~5.3]{BurFleHof13}).
}
\end{example}

\begin{example}[exponential decay of solution components]\label{ex:exponential} {\rm
If
$$\sum_{k=n+1}^\infty\vert x^\dagger_k\vert\leq K_1\,\exp(-n^\gamma)\quad\text{and}\quad
\sup_{\substack{a_k\in\{-1,0,1\}\\k=1,\ldots,n}}\left\Vert\sum_{k=1}^n a_k f^{(k)}\right\Vert_{Y^*}\leq K_2\,n^\nu$$
for $K_1,K_2\geq 0$ and $\gamma,\nu>0$, then the convergence rate obtained in Theorem~\ref{thm:improved} is
$$\|\xad-\xdag\|_{\1}= O\left(\delta\bigl(\log(1/\delta)\bigr)^{\frac{\nu}{\gamma}}\right) \qquad \mbox{as} \qquad \delta \to 0$$
(cf.\ \cite[Example~3.5]{BoHo13}).
}\end{example}

\begin{remark} \label{rem:distinguish} {\rm
Note that we always have weak convergence
\begin{equation} \label{eq:weakA}
Ae^{(k)}\rightharpoonup 0 \qquad \mbox{in}\;\;\;Y \quad \mbox{as} \qquad k\to\infty,
\end{equation}
because $\{e^{(k)}\}_{k\in\N}$ converges
weakly in $\2$ and $\widetilde A$ is weak-to-weak continuous since it is norm-to-norm continuous.
In \cite[Remark~2.5]{BurFleHof13} it was shown that the slightly stronger condition
\begin{equation} \label{eq:strongA}
\lim \limits_{k \to \infty} \|Ae^{(k)}\|_Y=0.
\end{equation}
enforces $\|f^{(k)}\|_{Y^\ast}\to\infty$ in Theorems~\ref{thm:BFH} and \ref{thm:improved}.

Obviously, condition (\ref{eq:strongA}) is satisfied if the underlying operator $\widetilde{A}$ is compact since compact
operators map weakly convergent sequences to norm convergent ones (note that this property is equivalent to compactness of $\widetilde{A}$ if
$X$ is a Hilbert or at least a reflexive Banach space, cf.\ \cite[Thm.~3.4.37]{Megginson98}).
On the other hand, one easily finds examples for noncompact operators which do not satisfy (\ref{eq:strongA}).
Choose, e.g., $X=Y=\2$ and let $\widetilde{A}$ be the identity. Then $\|Ae^{(k)}\|_Y=1$ for all $k \in \N$.
The question arises whether (\ref{eq:strongA}) is equivalent to compactness of $\widetilde A$.
The answer is `no' as the following example demonstrates.}
\end{remark}

\begin{example}[Ces\`aro operator]\label{ex:cesaro}{\rm
Let $X=Y=\2$ and define $\widetilde A:\2\to\2$ by
\begin{equation}\label{eq:cesaro}
[\widetilde Ax]_n=\frac{1}{n}\sum_{k=1}^n x_k.
\end{equation}
This operator is injective and noncompact with nonclosed range (see \cite[Solution 177]{Halmos82} or \cite{BroHalShi65}), but we have
$$\|Ae^{(k)}\|_{\2}^2=\sum_{n=k}^\infty\frac{1}{n^2}\to 0\quad\text{if}\quad k\to\infty.$$
Since with $f^{(1)}:=e^{(1)}$ and $f^{(k)}:=ke^{(k)}-(k-1)e^{(k-1)}$ for $k\geq 2$ assumption (\ref{eq:allrange}) of Theorems~\ref{thm:BFH}
and \ref{thm:improved} is satisfied, both convergence rates results apply to the specified operator.
\par
In the index function $\varphi_1$ in Theorem~\ref{thm:BFH} the second sum is
$$\sum_{k=1}^n\|f^{(k)}\|_{\2}=\sum_{k=1}^n\sqrt{(k-1)^2+k^2} \le \frac{n(n+1)}{\sqrt{2}}\,.$$
From
$$\sqrt{(k-1)^2+k^2}\geq\frac{1}{\sqrt{2}}(k-1+k)=\sqrt{2}k-\frac{1}{\sqrt{2}}$$
we even obtain a lower bound of the same order
$$\sum_{k=1}^n\|f^{(k)}\|_{\2}\geq\frac{\sqrt{2}n(n+1)}{2}-\frac{n}{\sqrt{2}}=\frac{n^2}{\sqrt{2}}.$$
On the other hand we now show that the supremum in the definition of $\varphi_2$ can be estimated above by
$$\sup_{\substack{a_k\in\{-1,0,1\}\\k=1,\ldots,n}}\left\Vert\sum_{k=1}^n a_k f^{(k)}\right\Vert_{\2}
\leq\frac{2}{\sqrt{3}}\,n^{3/2}.$$
At first we calculate
\begin{align*}
\sum_{k=1}^n a_k f^{(k)}
&=a_1e^{(1)}+\sum_{k=2}^nk\,a_ke^{(k)}-\sum_{k=1}^{n-1}k\,a_{k+1}e^{(k)}\\
&=n\,a_ne^{(n)}+\sum_{k=1}^{n-1}k\,(a_k-a_{k+1})e^{(k)}
\end{align*}
for arbitrary $a_1,\ldots,a_n$. Thus,
$$\left\Vert\sum_{k=1}^n a_k f^{(k)}\right\Vert_{\2}=\sqrt{n^2a_n^2+\sum_{k=1}^{n-1}k^2\,(a_k-a_{k+1})^2},$$
which attains its maximum over $(a_1,\ldots,a_n)\in\{-1,0,1\}^n$ for $a_k=(-1)^k$. Then
$$\sup_{\substack{a_k\in\{-1,0,1\}\\k=1,\ldots,n}}\left\Vert\sum_{k=1}^n a_k f^{(k)}\right\Vert_{\2}
=\sqrt{n^2+4\sum_{k=1}^{n-1}k^2}
=\sqrt{\frac{4}{3}n^3-n^2+\frac{2}{3}n}
\leq \frac{2}{\sqrt{3}}\,n^{3/2}.$$
Here, Examples~\ref{ex:hoelderrates} and \ref{ex:exponential} apply  and from these examples we see that the behaviour of the estimated sum in $\varphi_1$
and of the supremum in $\varphi_2$ directly carries over to the convergence rate. Thus, the slower growth of the supremum in
comparison to the faster growth of the sum yields a better rate for $\nu=\frac{3}{2}$ based on Theorem~\ref{thm:improved} than for $\nu=2$ based on Theorem~\ref{thm:BFH}.
}\end{example}

\begin{example}[diagonal operator]\label{ex:diagonal}{\rm
For a comparison we briefly recall Example~2.6 from \cite{BurFleHof13}, where $\widetilde A:X \to Y$ is a
{\sl compact diagonal operator}
between the separable Hilbert spaces $X$ and $Y$ with the singular system
$\{\sigma_k,u^{(k)},v^{(k)}\}_{k \in \N}$
and $\widetilde A u^{(k)}=\sigma_k\,v^{(k)},\;k \in \N$. Then the decay rate of the singular values $\sigma_k \to 0$ for $k \to \infty$ characterizes the degree of ill-posedness
of the equation (\ref{eq:tildeopeq}). For $\sigma_k \sim k^{-\zeta}, \;\zeta>0$, we have $\|Ae^{(k)}\|_Y=\|\widetilde A u^{(k)}\|_Y=\sigma_k \sim k^{-\zeta}$. The link condition (\ref{eq:allrange}) is satisfied
with $f^{(k)}=\frac{1}{\sigma_k}v^{(k)}$ and $\|f^{(k)}\|_Y \sim k^\zeta \to \infty$ as $k \to \infty$. Moreover, we have with some constant $C>0$
$$\sup_{\substack{a_k\in\{-1,0,1\}\\k=1,\ldots,n}}\left\Vert\sum_{k=1}^n a_k f^{(k)}\right\Vert_{Y} \,= \,\sqrt{\sum \limits_{k=1}^n \frac{1}{\sigma_k^2}} \,\le \, \sum \limits_{k=1}^n \frac{1}{\sigma_k}\, \le \, C\,n^{\zeta+1}, $$
hence $\nu=\zeta+1>1$ in Examples~\ref{ex:hoelderrates} and \ref{ex:exponential} based on Theorem~\ref{thm:improved}. Note that the values $\zeta>\frac{1}{2}$ and consequently $\nu>\frac{3}{2}$ correspond with the case
of Hilbert-Schmidt operators $\widetilde A$ and $\nu=\frac{3}{2}$ occurring in Example~\ref{ex:cesaro} is just a borderline case with respect to that fact.
}
\end{example}

\section{Ill-posedness of type I and II} \label{s4}

As suggested by M.~Z.~Nashed in \cite{Nashed86} we distinguish two types of ill-posedness for linear operator equations in a Banach space setting. Again, our focus is on injective operators.

\begin{definition} \label{def:type}
Let $B: Z_1 \to Z_2$ be an injective and bounded linear operator mapping between the infinite dimensional Banach spaces $Z_1$ and $Z_2$.  Then the operator equation
\begin{equation} \label{eq:B}
B x=y
\end{equation}
is called {\sl well-posed} if the range $\mathcal{R}(B)$ is a closed subset of $Z_2$, consequently {\sl ill-posed} if the range is not closed, i.e.~ $\mathcal{R}(B) \not= \overline{\mathcal{R}(B)}^{Z_2}$.
\par
In the ill-posed case, the equation (\ref{eq:B}) is called {\sl ill-posed of type I} if the range $\mathcal{R}(B)$ contains an {\sl infinite dimensional closed subspace}, and it is called {\sl ill-posed of type~II} otherwise.
\end{definition}

The two types of ill-posedness differ in the behavior of corresponding regularizers (cf.~\cite{Nashed86}) and with respect to smoothing properties of the linear operators $B$. If $B:=B_1: Z_1 \to Z_2$ is such that equation (\ref{eq:B})
proves to be ill-posed of type~I and $B:=B_2: Z_1 \to Z_2$ is such that equation (\ref{eq:B}) proves to be ill-posed of type~II, then $B_2$ tends to be `more smoothing' than $B_1$. Namely, a range inclusion $\mathcal{R}(B_2) \subset \mathcal{R}(B_1)$ may occur, but  $\mathcal{R}(B_1) \subset \mathcal{R}(B_2)$ cannot apply. We refer to \cite{BHTY06} for consequences of range inclusions and in particular to Example~10.2 ibidem for the interplay
of operators which characterize different types of ill-posedness.

The following proposition shows that at least for operators $B$ between Hilbert spaces $Z_1$ and $Z_2$ the type of
ill-posedness is determined by compactness properties of $B$.

\begin{proposition}\label{pro:Nashed}
If the operator equation (\ref{eq:B}) is well-posed or ill-posed of type~I, then the operator $B$ is non-compact.
Consequently, compactness of B implies ill-posedness of type~II.
\par
If $Z_1$ and $Z_2$ in equation (\ref{eq:B}) are Hilbert spaces and the equation is ill-posed,
then the equation is
ill-posed of type II if and only if $B$ is compact.
\end{proposition}
\begin{proof}
If the operator equation (\ref{eq:B}) is well-posed or ill-posed of type~I, there is an infinite dimensional Banach space $\hat Z_2$ included in the subspace $\mathcal{R}(B)$ of $Z_2$ with the same norm as in $Z_2$.
The preimage $\hat Z_1:=B^{-1}(\hat Z_2)$ is a Banach space included in $Z_1$ with the same norm as in $Z_1$. For a compact operator $B$, also its restriction $B|_{\hat Z_1}: \hat Z_1 \to \hat Z_2$ would be compact and moreover surjective. This contradicts the fact that a compact operator has only a closed range if it has a finite dimensional range, being a consequence of the non-compactness of the unit ball in a infinite dimensional
Banach space.
\par
For Hilbert spaces $Z_1$ and $Z_2$ and ill-posed equations (\ref{eq:B}), the equivalence of ill-posedness of type~I and the non-compactness of $B$ is well-known (cf.~\cite[Thm.~4.6]{Nashed86} and \cite[Lemma~5.8 and Theorem~5.9]{Doug98}).
\end{proof}

If compactness of $B$ is replaced by strict singularity, the characterization of ill-posedness types can be
made more precise for injective operators $B$.

\begin{definition}\label{def:strictsingularity}
A bounded linear operator $B$ between Banach spaces $Z_1$ and $Z_2$ is \emph{strictly singular} if its restriction to an infinite dimensional
subspace is never an isomorphism.
\end{definition}

\begin{proposition}\label{pro:Nashedsingular}
Let $B$ be an injective bounded linear operator between Banach spaces $Z_1$ and $Z_2$.
Then equation~(\ref{eq:B}) is ill-posed of type II if and only if $B$ is strictly singular.
\end{proposition}

\begin{proof}
We show that there exists an isomorphic restriction of $B$ to an infinite dimensional subspace of $Z_1$ if
and only if $\mathcal{R}(B)$ contains a closed infinite dimensional subspace.
\par
Obviously, if $Z_3$ is a closed infinite dimensional subspace of $\mathcal{R}(B)$, then the corresponding preimage
is of infinite dimension and the restriction of $B$ to this preimage is an isomorphism.
On the other hand, if there is an isomorphic restriction to an infinite dimensional subspace of $Z_1$, then
its image is also of infinite dimension and closed.
\end{proof}

Now we are going to apply Definition~\ref{def:type} to the equations (\ref{eq:tildeopeq}) and  (\ref{eq:opeq}) and to interpret the different cases. First we distinguish in the subsequent remark the possible cases arising in the context of equation (\ref{eq:tildeopeq}).

\begin{remark} \label{rem:opeq} {\rm
\begin{itemize}
\item[]
\item[(a)]
\textbf{Well-posed case}: The equation (\ref{eq:tildeopeq}) can be well-posed, which takes place if $\widetilde A$ is normally solvable. Linear Volterra integral {\sl equations of the second kind} as well as more generally
linear Fredholm integral equations of the second kind with appropriate kernels represent typical
examples of this case, where $X=Y=L^2(\Omega)$ with some bounded and sufficiently regular domain $\Omega$ in $\R^l,\;l=1,2,...\,,$ and the operator $\widetilde A$ is of the form $\widetilde A = I-K$ with the identity operator $I$ and a compact operator $K$ such that zero does not belong to the spectrum of the operator $\widetilde A$.
Normal solvability also occurs if $X=Y$ and $\widetilde A=I$. Then solving
(\ref{eq:tildeopeq}), for given noisy data $\yd \in Y$, is the simplest case of a {\sl denoising} problem (cf.~Section~\ref{sec:denoise}).
\item[(b)]
\textbf{Ill-posed case}: The equation (\ref{eq:tildeopeq}) is ill-posed if the operator $\widetilde A$  fails to be normally solvable. This is just the case if the inverse $\widetilde A^{-1}: \mathcal{R}(\widetilde A) \subset Y \to X$  is unbounded.
\begin{itemize}
\item[(b1)\,]
\textbf{Type I}:
Operators $\widetilde A$ for equations (\ref{eq:tildeopeq}) which prove to be ill-posed of type I are non-compact and even not strictly singular. If $Y$ is also a Hilbert space, all ill-posed equations (\ref{eq:tildeopeq}) with non-compact operator $\widetilde A$ are of this type.  Multiplication operators in  $X=Y=L^2(a,b)$ with $L^\infty(a,b)$-multiplier functions possessing essential zeros and linear convolution
operators in $X=Y=L^2(\R^l)$ with square-integrable kernels represent typical examples for this case. Furthermore, the Hausdorff moment problem with $\widetilde A:L^2(0,1) \to \2$ (cf.~\cite[Example~3.2]{Hof99}) is of this type.
\item[(b2)\,]
\textbf{Type II, non-compact}:
If $Y$ is not a Hilbert space then there exist strictly singular operators $\widetilde A$ with nonclosed range which are not compact, for example the
embedding operators from $\2$ to $\ell^q$ with $2<q<\infty$ (cf.~\cite[Theorem (a)]{GoldThorp63}). This also leads to ill-posed equations (\ref{eq:tildeopeq}) of type~II. Such operators  $\widetilde A$ can have a
non-separable range $\mathcal{R}(\widetilde A)$ (cf.~\cite[Remark on p.~335]{GoldThorp63}).
\item[(b3)\,]
\textbf{Type II, compact}:
The equation (\ref{eq:tildeopeq}) is ill-posed of type II if $\widetilde A: X \to Y$ is a {\sl compact} operator. Then $\widetilde A$ is strictly singular and the range $\mathcal{R}(\widetilde A)$ is a separable space. Typical examples with compact operators $\widetilde A$ are linear Fredholm and Volterra integral {\sl equations of the first kind} with square-integrable kernels in $L^2$-spaces $X$ and $Y$ over bounded and sufficiently regular domains in $\R^l$. Moreover, all bounded linear operators $\widetilde A: \2 \to \ell^q$ are compact for $1 \le q <2$ as was mentioned in \cite{Tar72}. In particular, if also $Y$ is a Hilbert space, then vice versa ill-posedness of type~II requires compactness of $\widetilde A$.
\end{itemize}
\end{itemize}
}\end{remark}

The diagram in Figure~1 illustrates the different cases in Remark~\ref{rem:opeq} and the results of this section concerning the relations between compactness, strict singularity and type of ill-posedness
for injective forward operators. By the way, we should mention that there exist non-injective strictly singular operators possessing a range which contains an
infinite dimensional closed subspace (cf.~\cite[first example]{GoldThorp63}).

\begin{figure}[h]
\begin{center}
\begin{picture}(0,0)%
\includegraphics{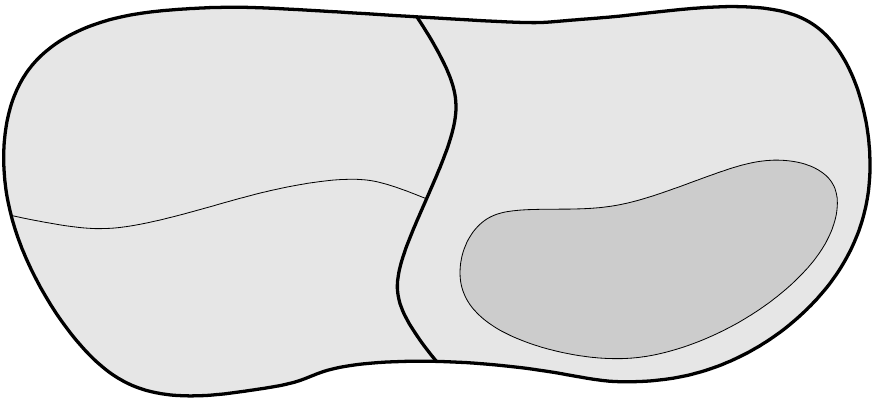}%
\end{picture}%
\setlength{\unitlength}{2072sp}%
\begingroup\makeatletter\ifx\SetFigFont\undefined%
\gdef\SetFigFont#1#2#3#4#5{%
  \reset@font\fontsize{#1}{#2pt}%
  \fontfamily{#3}\fontseries{#4}\fontshape{#5}%
  \selectfont}%
\fi\endgroup%
\begin{picture}(7986,3626)(2491,-5360)
\put(8191,-4246){\makebox(0,0)[b]{\smash{{\SetFigFont{8}{9.6}{\familydefault}{\mddefault}{\updefault}{\color[rgb]{0,0,0}operator compact}%
}}}}
\put(4481,-3493){\rotatebox{13.0}{\makebox(0,0)[b]{\smash{{\SetFigFont{8}{9.6}{\familydefault}{\mddefault}{\updefault}{\color[rgb]{0,0,0}operator not}%
}}}}}
\put(4574,-3785){\rotatebox{13.0}{\makebox(0,0)[b]{\smash{{\SetFigFont{8}{9.6}{\familydefault}{\mddefault}{\updefault}{\color[rgb]{0,0,0}strictly singular}%
}}}}}
\put(7839,-2998){\makebox(0,0)[b]{\smash{{\SetFigFont{8}{9.6}{\familydefault}{\mddefault}{\updefault}{\color[rgb]{0,0,0}operator strictly}%
}}}}
\put(7852,-3273){\makebox(0,0)[b]{\smash{{\SetFigFont{8}{9.6}{\familydefault}{\mddefault}{\updefault}{\color[rgb]{0,0,0}singular}%
}}}}
\put(4590,-2632){\makebox(0,0)[b]{\smash{{\SetFigFont{8}{9.6}{\rmdefault}{\bfdefault}{\updefault}{\color[rgb]{0,0,0}well-posed}%
}}}}
\put(8438,-2450){\makebox(0,0)[b]{\smash{{\SetFigFont{8}{9.6}{\rmdefault}{\bfdefault}{\updefault}{\color[rgb]{0,0,0}ill-posed of type II}%
}}}}
\put(4653,-4638){\makebox(0,0)[b]{\smash{{\SetFigFont{8}{9.6}{\rmdefault}{\bfdefault}{\updefault}{\color[rgb]{0,0,0}ill-posed of type I}%
}}}}
\end{picture}%
\end{center}
\caption{Relations between strict singularity, compactness and type of ill-posedness for equations~(\ref{eq:B}) with injective bounded linear operator.}
\end{figure}
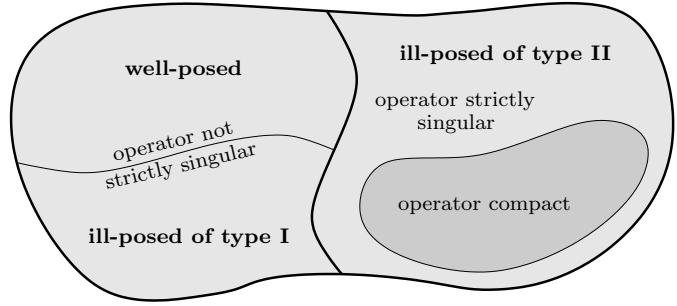

A scenario completely different from Remark~\ref{rem:opeq} occurs for equation (\ref{eq:opeq}) due to the composition structure (\ref{eq:three}) of the operator $A$. The fact that the non-compact embedding operator $\mathcal{E}_2$ is strictly
singular (cf.~\cite[Theorem]{GoldThorp63})  prevents the occurrence of well-posedness and ill-posedness of type~I in the context of this equation.

\begin{proposition} \label{pro:onlyII}
Under the assumptions stated above equation (\ref{eq:opeq}) is \linebreak always ill-posed of type~II.
\end{proposition}
\begin{proof}
Taking into account Proposition~\ref{pro:Nashedsingular} we only have to show that $A$ is always strictly singular.
But this follows immediately from the composition structure (\ref{eq:three}) and the two facts that
$\mathcal{E}_2$ is strictly singular (see~\cite{GoldThorp63}) and that the composition of a strictly singular operator with a bounded linear
operator is again strictly singular.
\end{proof}

\section{The special case of denoising} \label{sec:denoise}

Finally, we apply our results to a typical denoising problem. Given a noisy signal one wants to remove the noise.
For this purpose one decomposes the signal with respect to a wavelet basis (or any other orthonormal system) and tries to
find a sparse approximation with respect to this basis. Thus, in our setting we choose $\widetilde A$ to be the identity
on $\2$. Since in some applications it might be reasonable to measure the noise in a weaker norm we extend $Y$ to $\ell^q$
with $2\leq q\leq\infty$. Then $A:=\mathcal{E}_q$ is the embedding of $\1$ into $\ell^q$. In the sequel we only look at $A$ and therefore extend
the feasible values for $q$ to $1\leq q\leq\infty$.

The minimization problem (\ref{eq:TikBanach}) now reads as
\begin{equation} \label{eq:Tiklq}
\frac{1}{p}\|\mathcal{E}_q x-y^\delta\|_{\ell^q}^p + \alpha \,\|x\|_{\1}
\to \min, \quad \mbox{subject to} \quad x  \in \1,
\end{equation}
where the exact signal $x^\dagger$ is assumed to be nearly sparse, i.e.\ $x^\dagger\in\1$.
From the computational point of view it seems to be helpful to apply $p:=q$ for the exponent in the misfit term of (\ref{eq:Tiklq}) whenever $1<q<\infty$.
If we measure the error after denosing in the $\1$-norm as $\|\xad-\xdag\|_{\1}$, then the chances of having small errors improve with decreasing values $q$, since the strength of the norm in $Y$ grows if $q$ decreases.

\begin{proposition} \label{pro:lq}
For the embedding operator $A=\mathcal{E}_q$ from $\1$ to $\ell^q$ with $1<q \le \infty$ equation (\ref{eq:opeq}) is ill-posed of type~II.
We have weak convergence $Ae^{(k)} \rightharpoonup 0$ if $k\to\infty$ for $1<q \le \infty$, but no convergence in norm.
For all $1\leq q \le  \infty$ and all $k \in \N$ the link condition (\ref{eq:allrange}) is satisfied with $f^{(k)}=e^{(k)}$.
Thus, Theorems~\ref{thm:BFH} and \ref{thm:improved} apply and the corresponding index functions are
\begin{equation} 
\varphi_1(t)=2\inf_{n\in\mathbb{N}}\left(\sum_{k=n+1}^\infty\vert x^\dagger_k\vert
+t\,n\right)\qquad \quad \mbox{if} \quad 1\leq q\leq\infty
\end{equation}
and
\begin{equation} \label{eq:phiimp3}
\varphi_2(t)=2\inf_{n\in\mathbb{N}}\left(\sum_{k=n+1}^\infty\vert x^\dagger_k\vert+t\,n^\theta\right)\qquad
\end{equation}
with
$$\theta=\begin{cases}
1-\frac{1}{q},&\text{if}\quad 1\leq q<\infty,\\
1,&\text{if}\quad q=\infty.
\end{cases}$$
\end{proposition}

\begin{proof}
The ill-posedness of type~II is a consequence of Proposition~\ref{pro:onlyII} whenever $2 \le q \le \infty$, because the injective and bounded embedding operator from $\ell^2$ to $\ell^q$ plays here the role of $\widetilde A$. On the other hand, the non-existence of an infinite dimensional closed subspace in $\ell^q$ for $1<q<2$ included in $\1$ also follows from the theorem in \cite{GoldThorp63}, since the corresponding embedding operators are strictly singular.
\par
It is evident that $e^{(k)}$ does not converge to zero in the norm of $\ell^q$, but we have $e^{(k)} \rightharpoonup 0$ in $\ell^q$ for all $1<q \le \infty$. Also the validity of (\ref{eq:allrange}) is evident.
\par
The index functions $\varphi_1$ and $\varphi_2$ in the convergence rates theorems can be computed easily for the special case
under consideration.
\end{proof}

\begin{example}[H\"older rates] {\rm If the decay rate  $\xdag_k \to 0$ as $k \to \infty$ of the remaining solution coefficients is of power type
\begin{equation} \label{eq:lqdecay}
|\xdag_k| \le C_{\xdag}\,k^{\mu-1},\; k \in \N, \quad\text{or equivalently}\quad\sum_{k=n+1}^\infty\vert x^\dagger_k\vert\leq K_{\xdag}\,n^{-\mu},\; n \in \N,
\end{equation}
with constants $\mu>0$ and $C_{\xdag},K_{\xdag}>0$, we immediately derive from Proposition~\ref{pro:lq} and Example~\ref{ex:hoelderrates}
the H\"older convergence rates for the denoising problem with forward operator $A=\mathcal{E}_q$ as
\begin{equation} \label{eq:lqrate}
\|\xad-\xdag\|_{\1}= O\left(\delta^{\frac{\mu}{\mu+1-\frac{1}{q}}}\right) \quad \mbox{as} \quad \delta \to 0 \qquad \mbox{if} \quad 1\leq q<\infty
\end{equation}
and
\begin{equation} \label{eq:lqrateinf}
\|\xad-\xdag\|_{\1}= O\left(\delta^{\frac{\mu}{\mu+1}}\right) \quad \mbox{as} \quad \delta \to 0 \qquad \mbox{if} \qquad q=+\infty.
\end{equation}
As expected the rate grows if $q$ decreases, i.e., if the noise is measured in a stronger norm.
On the other hand, the borderline case $q=1$ leads to a well-posed equation (\ref{eq:opeq}).
In this case the index function $\varphi_2$ attains the form
$$\varphi_2(t)=2\inf_{n\in\mathbb{N}}\left(\sum_{k=n+1}^\infty\vert x^\dagger_k\vert
+t \right)=2t\qquad \quad \mbox{if} \quad q =1$$
and the corresponding rate is
$$ \|\xad-\xdag\|_{\1}= O\left(\delta\right) \quad \mbox{as} \quad \delta \to 0 \qquad \mbox{if} \qquad q=1,$$
which is typical for well-posed situations.
}\end{example}

The example also shows that the improved index function $\varphi_2$ from Theorem~\ref{thm:improved} indeed provides a better convergence rate for
$1\leq q<\infty$ than the original index function $\varphi_1$ from Theorem~\ref{thm:BFH}.
Note that $\varphi_1$ in Proposition~\ref{pro:lq} is for all $q$ the same function as $\varphi_2$ with $q=\infty$.

At the end we should mention that the rate results (\ref{eq:lqrate}) and (\ref{eq:lqrateinf}) yield the values $0<\nu \le 1$ in formula (\ref{firstrate}) from  Example~\ref{ex:hoelderrates}. Consequently,
the H\"older rates in Examples~\ref{ex:cesaro} and \ref{ex:diagonal} with $\nu>1$ are always lower than the observed rates for the denoising case, which indicates a lower degree of ill-posedness for the denoising problem.

\subsection*{Acknowledgement}
The authors very appreciate the fruitful discussion with Radu I.~Bo\c{t} (University of Vienna) and are particularly grateful that he brought the paper \cite{GoldThorp63} to our attention. We also express our thanks to Peter Stollmann and Thomas Kalmes (TU Chemnitz) for valuable hints.
Jens Flemming and Bernd Hofmann were supported by the German Research Foundation (DFG) under grants FL~832/1-1 and  HO~1454/8-2, respectively. Ivan Veseli\'c was supported by  the PPP-grant 56266051 of the DAAD and the Ministry of Science of the Republic of Croatia.

%
%

\end{document}